\documentclass{amsart}
\title{Resolving Irreducible $\mathbb{C}S_n$-Modules by Modules Restricted from $GL_n(\mathbb{C})$}
\author{Christopher Ryba}
\address{Department of Mathematics, Massachusetts Institute of Technology, Cambridge, MA 02139, USA}
\email{ryba@mit.edu}
\usepackage{hyperref}
\usepackage{fullpage}
\usepackage{amsmath}
\usepackage{amsfonts}
\usepackage{amssymb}
\usepackage{comment}
\usepackage{ytableau}
\usepackage{subfig}
\usepackage{todonotes}
\usepackage{mathtools}
\usepackage{graphicx, caption}

\begin{document}
\maketitle

\newtheorem{theorem}{Theorem}[section]
\newtheorem{lemma}[theorem]{Lemma}
\newtheorem{proposition}[theorem]{Proposition}
\newtheorem{corollary}[theorem]{Corollary}

\begin{abstract}
We construct a resolution of irreducible complex representations of the symmetric group $S_n$ by restrictions of representations of $GL_n(\mathbb{C})$ (where $S_n$ is the subgroup of permutation matrices). This categorifies a recent result of Assaf and Speyer. Our construction also gives minimal resolutions of simple $\mathcal{F}$-modules (here $\mathcal{F}$ is the category of finite sets).
\end{abstract}

\section{Introduction}
\noindent
The symmetric group $S_n$ may be viewed as the subgroup of the general linear group $GL_n(\mathbb{C})$ consisting of permutation matrices. We may therefore consider the restriction to $S_n$ of irreducible $GL_n(\mathbb{C})$ representations. Let $S^\lambda$ denote the irreducible representation of $\mathbb{C}S_n$ indexed by the partition $\lambda$ (so necessarily $n$ is the size of $\lambda$). Let $\mathbb{S}^\lambda$ denote the Schur functor associated to a partition $\lambda$, so that $\mathbb{S}^\lambda(\mathbb{C}^n)$ is an irreducible representation of $GL_n(\mathbb{C})$, provided that $l(\lambda) \leq n$. Let us write $[M]$ for the image of a module in the Grothendieck ring of $\mathbb{C}S_n$-modules. Thus, the restriction multiplicities $a_{\mu}^{\lambda}$ are defined via
\[
[\mbox{Res}_{S_n}^{GL_n}(\mathbb{S}^\lambda(\mathbb{C}^n)] = \sum_{\mu \vdash n} a_{\mu}^{\lambda}[S^\mu].
\]
Although a positive combinatorial formula for the restriction multiplicities is not currently known, there is an expression using plethysm of symmetric functions (see \cite{macdonald} Chapter 1 Section 8 for background about plethysm). Let us write $s_\lambda$ for the Schur functions (indexed by partitions $\lambda$). The complete symmetric functions, $h_n$, are the Schur functions indexed by the one-part partitions $(n)$. We recall the Schur functions $s_\lambda$ form an orthonormal basis of the ring of symmetric functions with respect to the usual inner product, denoted $\langle -,- \rangle$ (see \cite{macdonald} Chapter 1 Section 4). Let $f[g]$ denote the plethysm of a symmetric function $f$ with another symmetric function $g$. Then,
\[
a_\mu^\lambda =  \langle s_\lambda, s_\mu[1 + h_1 + h_2 + \cdots] \rangle,
\]
see \cite{gay} or Exercise 7.74 of \cite{stanley}. We will need to consider the Lyndon symmetric function,
\[
L_n = \frac{1}{n} \sum_{d\mid n}\mu(d)p_d^{n/d},
\]
where $\mu(d)$ is the M\"{o}bius function and $p_d$ is the $d$-th power-sum symmetric function. It is important for us that $L_n$ is the $GL(V)$ character of the degree $n$ components of the free Lie algebra on $V$ (see the first proof of Theorem 8.1 of \cite{reutenauer}, which proves this to deduce a related result). For convenience we define the total Lyndon symmetric function $L = L_1 + L_2 + \cdots$; this is the character of the (whole) free Lie algebra on $V$.
\newline \newline \noindent
Instead of asking for the restriction coefficients $a_\mu^\lambda$, we may ask the inverse question: how can one express $[S^\mu]$ in terms of $[\mbox{Res}_{S_n}^{GL_n}(\mathbb{S}^\lambda(\mathbb{C}^n))]$? This question was recently answered by Assaf and Speyer in \cite{AS}. For a partition $\mu = (\mu_1, \mu_2, \ldots)$ of any size, let $\mu[n]$ denote $(n-|\mu|, \mu_1, \mu_2, \ldots)$ (a partition of $n$ provided that $n \geq |\mu|+ \mu_1$). Assaf and Speyer showed
\[
[S^{\mu[n]}] = \sum_{\lambda} b_{\lambda}^\mu [\mathbb{S}^\nu(\mathbb{C}^n)],
\]
where
\[
b_\lambda^\mu = (-1)^{|\mu| - |\lambda|}\sum_{\mu / \nu \mbox{ vert. strip}} 
\langle s_{\nu^\prime}, s_{\lambda^\prime}[L] \rangle.
\]
The notation $\mu / \nu \mbox{ vert. strip}$ means that the diagram of $\mu$ may be obtained from the diagram of $\nu$ by adding boxes, no two in the same row, and primes indicate dual partitions.
\newline \newline \noindent
It is more convenient to work with 
\[
M_n^{\mu}  = \mbox{Ind}_{S_{|\mu|} \times S_{n - |\mu|}}^{S_n} (S^\mu \boxtimes \mathbf{1}),
\]
which decompose into the irreducible $S^{\nu[n]}$ via the Pieri rule:
\[
[M_n^{\mu}] = \sum_{\mu / \nu \mbox{ horiz. strip}} [S^{\nu[n]}].
\]
Here, $\mu / \nu \mbox{ horiz. strip}$ means that the diagram of $\mu$ may be obtained from the diagram of $\nu$ by adding boxes, no two in the same column. The formula for $b_\lambda^\mu$ is equivalent to the following statement (see Theorem 3 and Proposition 5 of \cite{AS}):
\[
[M_n^{\mu}] = \sum_{\lambda} (-1)^{|\mu| - |\lambda|} \langle s_{\mu^\prime}, s_{\lambda^\prime}[L] \rangle [\mathbb{S}^\lambda(\mathbb{C}^n)].
\]
The purpose of this note is to give a a categorification of this answer, namely a (minimal) resolution of $M_n^\mu$ by restrictions of $\mathbb{S}^\lambda(\mathbb{C}^n)$; this is accomplished in Theorem \ref{final_resolution}. Along the way, this explains the presence of the character of the free Lie algebra in the formula, and constructs projective resolutions in the category of $\mathcal{F}$-modules (over $\mathbb{Q}$) introduced by Wiltshire-Gordon in \cite{JWG}.

\section*{Acknowledgements}
\noindent
The author would like to thank Gurbir Dhillon for helpful comments on this paper.

\section{The Resolution}
\noindent
We begin by calculating the cohomology of the free Lie algebra on a fixed vector space. Although this result is very well known, it is instrumental in what follows, so we include it for completeness.
\newline \newline \noindent
Let $L$ be the free Lie algebra on $V = \mathbb{C}^m$. Then $\mathfrak{g} = L^{\oplus n} = L \otimes \mathbb{C}^n$ is again a Lie algebra. It has an action of $S_n$ by permuting the $L$ summands, coming from an action of $GL_n(\mathbb{C})$ that does not respect the Lie algebra structure. We consider the Lie algebra cohomology of $\mathfrak{g}$ (with coefficients in the trivial module).
\newline \newline \noindent
Recall that the Lie algebra cohomology is $\mbox{Ext}_{U(\mathfrak{g})}^{i}(\mathbb{C}, \mathbb{C})$. We first consider the case for $t=1$, so $\mathfrak{g} = L$. Now $U(L)$ is just the tensor algebra of $V$, which we denote $T(V)$. We therefore have a (graded) free resolution
\[
0 \xrightarrow{} T(V) \otimes V \xrightarrow{d_1} T(V) \xrightarrow{d_0} \mathbb{C} \xrightarrow{} 0.
\]
Here, $d_1(x \otimes v) = xv$ (product in $T(V)$), while $d_0$ simply projects onto the degree zero component. Crucially, $GL(V)$ acts by automorphisms on $L$ (which was the free Lie algebra on $V$), and the above complex is equivariant for this action. The Lie algebra cohomology is given by the cohomology of the complex
\[
0 \leftarrow \hom_{T(V)}(T(V) \otimes V, \mathbb{C}) \xleftarrow{d_1^*} \hom_{T(V)}(T(V), \mathbb{C})  \leftarrow 0.
\]
We easily see the differential $d_1^*$ is zero because any element of $\hom_{T(V)}(T(V), \mathbb{C})$ is zero on a positive degree element of $T(V)$, but the image of $d_1$ is contained in degrees greater than or equal to $1$. We thus conclude that $H^0(L, \mathbb{C}) = \mathbb{C}$, and $H^1(L, \mathbb{C}) = V^*$, with all higher cohomology vanishing. Next, we obtain the Lie algebra cohomology of $\mathfrak{g} = L^{\oplus n}$.

\begin{proposition} \label{coh_prop}
For $0 \leq i \leq n$:
\[
H^i(\mathfrak{g}, \mathbb{C}) = \mbox{\emph{Ind}}_{S_i \times S_{n-i}}^{S_n}((V^*)^{\otimes i} \otimes \varepsilon_i \boxtimes \mathbb{C}^{\otimes (n-i)}),
\]
where $\varepsilon_i$ is the sign representation of $S_i$, and $\mathbb{C}^{\otimes (n-i)}$ is the trivial representation of $S_{n-i}$. Further, for $i > n$, the cohomology $H^i(\mathfrak{g}, \mathbb{C})$ vanishes.
\end{proposition}
\begin{proof}
We apply the K\"{u}nneth theorem in an $S_n$-equivariant way. The sign representation $\varepsilon_i$ arises because of the Koszul sign rule (cohomology is only graded commutative).
\end{proof}
\noindent
Now let us compute the Lie algebra cohomology of $\mathfrak{g}$ using the Chevalley-Eilenberg complex \cite{weibel}. Recall that the $i$-th cochain group is
\[
\hom_{\mathbb{C}}(\bigwedge\nolimits^i (\mathfrak{g}), \mathbb{C})
\]
and the differential $d$ is given by the formula
\[
d(f)(x_1, \ldots, x_{k+1}) = \sum_{i<j} (-1)^{j-i}f([x_i, x_j], x_1, \ldots, \hat{x}_i, \ldots, \hat{x}_j, \ldots , x_{k+1})
\]
where hats indicate omitted arguments. 
This differential is $GL(V) \times S_n$-equivariant and homogeneous in terms of the grading on $\mathfrak{g}$ (the grading corresponds to the action of $\mathbb{C}^\times = Z(GL(V))$).
\newline \newline \noindent
Note that $\mathfrak{g}$ is graded in strictly positive degrees. As an algebraic representation of $GL(V)$, the $i$-th cochain group,
\[
\hom_{R}(\bigwedge\nolimits^i (\mathfrak{g}), \mathbb{C})
\]
is contained in degrees $\leq -i$. This means that if we are interested only in the degree $-i$ component of the cohomology, we may truncate the Chevalley-Eilenberg complex after $i$ steps. Thus, if we write a subscript $-i$ to indicate the degree $-i$ component of an $GL(V)$ representation, we obtain the following.
\begin{proposition} \label{res_prop}
The complex (with differential inherited from the Chevalley-Eilenberg complex)
\[
0 \leftarrow\hom_{\mathbb{C}}(\bigwedge\nolimits^i (\mathfrak{g}), \mathbb{C})_{-i} \leftarrow \hom_{\mathbb{C}}(\bigwedge\nolimits^{i-1} (\mathfrak{g}), \mathbb{C})_{-i} \leftarrow \cdots \leftarrow \hom_{\mathbb{C}}(\bigwedge\nolimits^1 (\mathfrak{g}), \mathbb{C})_{-i} \leftarrow \hom_{\mathbb{C}}(\bigwedge\nolimits^0 (\mathfrak{g}), \mathbb{C})_{-i} \leftarrow 0
\]
has cohomology $H^i(\mathfrak{g}, \mathbb{C})$ on the far left, and zero elsewhere.
\end{proposition}

\section{Resolving the $\mathbb{C}S_n$-modules $M_n^\mu$} \label{sec3}
\noindent
Let us take the multiplicity space of the $GL(V)$-irreducible $\mathbb{S}^{\mu^\prime}(V^*)$.
\begin{proposition} \label{next_prop}
The $\mathbb{S}^{\mu^\prime}(V^*)$ multiplicity space in the cohomology $H^i(\mathfrak{g}, \mathbb{C})$ is 
\[
M_n^{\mu} = \mbox{\emph{Ind}}_{S_i \times S_{n-i}}^{S_n}({S}^\mu \boxtimes \mathbf{1})
\]
\end{proposition}
\begin{proof}
We apply Schur-Weyl duality to Proposition \ref{coh_prop}, noting that $S^\lambda \otimes \varepsilon_i = S^{\lambda^\prime}$:
\[
H^i(\mathfrak{g}, \mathbb{C}) = \mbox{Ind}_{S_i \times S_{n-i}}^{S_n}((V^*)^{\otimes i} \otimes \varepsilon_i \boxtimes \mathbf{1}) = \mbox{Ind}_{S_i \times S_{n-i}}^{S_n}(\bigoplus_{\lambda \vdash i}\mathbb{S}^{\lambda}(V^*)\otimes S^{\lambda^\prime} \boxtimes \mathbf{1}).
\]
Hence, the $\mathbb{S}^{\mu^\prime}(V^*)$ multiplicity space is $\mbox{Ind}_{S_i \times S_{n-i}}^{S_n}({S}^\mu \boxtimes \mathbf{1})$.
\end{proof}
\noindent
Because the complex we constructed in Proposition \ref{res_prop} is $GL(V)$ equivariant, taking cohomology commutes with taking the $\mathbb{S}^{\mu^\prime}(V^*)$ multiplicity space. We immediately obtain the following.
\begin{theorem} \label{final_resolution}
Consider the complex of $S_n$ representations
\[
\hom_{GL(V)}\left(\mathbb{S}^{\mu^\prime}(V^*), \hom_{\mathbb{C}}(\bigwedge\nolimits^i (\mathfrak{g}), \mathbb{C})\right)
\]
for $|\mu| \geq i \geq 0$ with maps induced by the differential of the Chevalley-Eilenberg complex. This is a resolution of $M_n^\mu$ by representations restricted from $GL_n(\mathbb{C})$.
\end{theorem}
\begin{proof}
This is immediate from Proposition \ref{next_prop} and Proposition \ref{res_prop}. 
\end{proof}
\noindent
Should we wish to resolve the irreducible $S^\mu$, rather than $M_n^{\mu}$, we simply take $n = |\mu|$ so that $M_n^\mu = S^\mu$.
\newline \newline \noindent
We now take the Euler characteristic of our complex, viewed as an element of the Grothendieck ring of $\mathbb{C}S_n$-modules tensored with the Grothendeick ring of $GL(V)$-modules; we view the latter as the ring of symmetric functions. In the language of symmetric functions, the Schur function $s_\lambda$ corresponds to the irreducible representation $\mathbb{S}^\lambda(V)$ (strictly speaking, we must quotient out $s_\lambda$ for $\lambda$ with more parts than $m = \dim(V)$, but this will never be an issue). We express the cohomology groups in terms of symmetric functions; as in the proof of Proposition \ref{next_prop}, Schur-Weyl duality gives
\[
H^i(\mathfrak{g}, \mathbb{C}) = \mbox{Ind}_{S_i \times S_{n-i}}^{S_n}(\bigoplus_{\lambda \vdash i}\mathbb{S}^{\lambda}(V^*)\otimes S^{\lambda^\prime} \boxtimes \mathbf{1}).
\]
Letting $\lambda = \mu^\prime$ and passing to Grothendieck rings, this becomes $\sum_{\mu \vdash i} s_{\mu^\prime}(x^{-1}) [M_n^\mu]$, where a Schur function indicates a representation of $GL(V)$ (the inverted variables account for the dualised space $V^*$). 
Calculating the Euler characteristic directly from the cochain groups, we consider the $i$-th exterior power of $\mathfrak{g} = L \otimes \mathbb{C}^n$,
\begin{equation} \label{chain_form}
\bigwedge\nolimits^i (\mathfrak{g}) = \bigoplus_{\lambda \vdash i} \mathbb{S}^{\lambda^\prime}(L) \otimes \mathbb{S}^\lambda(\mathbb{C}^n)
\end{equation}
which gives $\sum_{\lambda \vdash i} [S^\lambda(\mathbb{C}^n)] s_{\lambda^\prime}[L](x)$. The actual chain groups are the duals of these exterior powers, so we replace the symmetric function variables $x$ with their inverses $x^{-1}$. When we introduce a factor of $(-1)^{|\lambda|-|\mu|}$ from the signs in the Euler characteristic, we obtain
\[
\sum_{\lambda} (-1)^{|\lambda|-|\mu|}  [\mathbb{S}^\lambda(\mathbb{C}^n)] s_{\lambda^\prime}[L](x^{-1}).
\]
Thus the coefficient of $[\mathbb{S}^\lambda(\mathbb{C}^n)]$ in $[M_{\mu}^n]$ is the coefficient of $s_{\mu^\prime}(x^{-1})$ in $(-1)^{|\mu|-|\lambda|}s_{\lambda^\prime}[L](x^{-1})$, which gives us
\[
[M_\mu^n] = \sum_{\lambda} (-1)^{|\mu|-|\lambda|} [\mathbb{S}^\lambda(\mathbb{C}^n)] \langle s_{\lambda^\prime}[L], s_{\mu^\prime}\rangle .
\]
This provides an alternative proof the formula from \cite{AS} for expressing the irreducible representation $S^{\mu[n]}$ of $S_n$ in terms of restrictions $\mbox{Res}_{S_n}^{GL_n(\mathbb{C})}(\mathbb{S}^{\lambda}(\mathbb{C}^{n}))$. This construction addresses a remark of Assaf and Speyer by explaining the presence of the character of the free Lie algebra (namely, $L$) in the formula.
\section{Application to $\mathcal{F}$-modules}
\noindent
Let $\mathcal{F}$ denote the category of finite sets. An $\mathcal{F}$-module is a functor from $\mathcal{F}$ to vector spaces over a fixed field. These were introduced in \cite{JWG}, and their homological algebra was studied over $\mathbb{Q}$.
An $\mathcal{F}$-module consists of a $S_n$-module for each $n$ together with suitably compatible maps between them. (This is because the image of an $n$-element set carries an action of $Aut(\{1,2,\ldots,n\}) = S_n$.) When $\mu$ is a partition different from $(1^k)$ (i.e. not a single column), $M_n^\mu$ (considered for fixed $\mu$ but varying $n$) defines an irreducible $\mathcal{F}$-module, by demanding that an $n$-element set in $\mathcal{F}$ map to $M_n^\mu$ (see Theorem 5.5 of \cite{JWG}). Furthermore, in this category, objects obtained by restricting $\mathbb{S}^{\lambda}(\mathbb{Q}^n)$ to $S_n$ are projective (see Definition 4.8 and Proposition 4.12 of \cite{JWG}). Our resolution (provided we replaces all instances of $\mathbb{C}$ with $\mathbb{Q}$) therefore gives a projective resolution of these simple $\mathcal{F}$-modules $M_n^\mu$. This resolution is in fact minimal (in the sense that each step in the projective resolution is as small as possible). This follows from the following two facts. Firstly, the $r$-th term in the resolution of $M_n^\mu$ is a sum of $\mbox{Res}_{S_n}^{GL_n(\mathbb{Q})}(\mathbb{S}^{\lambda}(\mathbb{Q}^n))$ with $|\lambda| = |\mu|-r$, which is a consequence of Equation \ref{chain_form}. In particular, such a module with fixed $\lambda$ can only appear in one step of the resolution. Secondly, a theorem of Littlewood (Theorem XI of \cite{lw}), states that the restriction multiplicity $a_{\mu}^\lambda$ is equal to $\delta_{\mu, \lambda}$ if $|\mu| \geq |\lambda|$. Thus, $[\mbox{Res}_{S_n}^{GL_n(\mathbb{Q})}(\mathbb{S}^{\lambda}(\mathbb{C}^n)]$ are linearly independent elements of the Grothendieck ring of $S_n$-modules, provided $n$ is sufficiently large. Furthermore, the $[\mbox{Res}_{S_n}^{GL_n(\mathbb{Q})}(\mathbb{S}^{\lambda}(\mathbb{C}^n)]$ should only occur in the resolution in order of decreasing $|\lambda|$ (as in our resolution). Together with Observation 4.25 of \cite{JWG}, which provides a projective resolution of certain $\mathcal{F}$-modules $D_k$ (which can be thought of as substitutes for $M_n^\mu$ when $\mu = (1^k)$), we obtain minimal projective resolutions of all finitely-generated $\mathcal{F}$-modules over $\mathbb{Q}$.

\bibliographystyle{alpha}
\bibliography{letter.bib}

\end{document}